\documentclass[a4paper,11pt]{article}

\usepackage{amsmath}
\usepackage{amsthm}
\usepackage{amsfonts}
\usepackage{amssymb}
\usepackage{color}
\usepackage{graphicx}
\usepackage{xspace}

\newcount\mwhr\newcount\mwmin\newcount\mwten\newcount\mwtemp
\mwhr=\number\time
\divide\mwhr by60

\mwten=\mwhr
\multiply\mwten by-60
\advance\mwten by\number\time{}
\divide\mwten by10

\mwmin=\mwhr
\multiply\mwmin by-60
\mwtemp=\mwten
\multiply\mwtemp by-10
\advance\mwmin by\mwtemp
\advance\mwmin by\number\time{}

\newcommand\cP{{\mathcal P}}

\newcommand\cB{{\mathcal B}}
\newcommand\cA{{\mathcal A}}
\newcommand\cC{{\mathcal C}}
\newcommand\eps{\varepsilon}
\newcommand\Prb{\mathbb{P}}
\newcommand\N{\mathbb{N}}

\newcommand{\whp}{whp\xspace}

\title{Sharpness in the $k$-nearest neighbours random geometric graph model}
\author{Victor Falgas--Ravry\footnote{School of Mathematical Sciences, Queen Mary,
University of London, London E1 4NS, England} \footnote{\tt v.falgas-ravry@qmul.ac.uk}\and Mark Walters$^*$\footnote{\tt m.walters@qmul.ac.uk}}

\begin{document}

\newtheorem{theorem}{Theorem}
\newtheorem{lemma}[theorem]{Lemma}
\newtheorem{corollary}[theorem]{Corollary}
\newtheorem*{quotedresult}{Lemma}
\newtheorem*{conjecture}{Conjecture}

\maketitle

%\thispagestyle{myheadings} % this sets the pagestyle for the page with the title as well

%made the following `sharpness in $k$'
\begin{abstract}
  Let $S_{n,k}$ denote the random geometric graph obtained by placing
  points in a square box of area $n$ according to a Poisson process of
  intensity $1$ and joining each point to its $k$ nearest neighbours.
  In~\cite{MR2135151} Balister, Bollob\'as, Sarkar and Walters conjectured
  that for every  $0< \varepsilon <1$ and all $n$
  sufficiently large there exists $C=C(\eps)$ such that if
 \[\mathbb{P}(S_{n,k} \mbox{ connected}) \geq \varepsilon\] 
then
 \[\mathbb{P}(S_{n,k+C} \mbox{ connected}) > 1 -\varepsilon .\]
In this paper we prove this conjecture.

As a corollary we prove that there is a constant $C'$ such that whenever
$k(n)$ is a sequence of integers with
\[
\mathbb{P}(S_{n,k(n)} \textrm{ connected})\rightarrow 1 \textrm{ as n
} \rightarrow \infty,
\]
then for any integer sequence $s(n)$ with $s(n)=o(\log n)$,
\[
\mathbb{P}(S_{n,k(n)+\lfloor C's\log \log n\rfloor} \textrm{
  $s$-connected})\rightarrow 1 \textrm{ as n } \rightarrow \infty.
\]
This proves another conjecture of Balister, Bollob\'as, Sarkar and Walters~\cite{MR2479805}.

\end{abstract}

\section*{Introduction}
Let $S_n$ be the square $[0,\sqrt{n}]\times[0,\sqrt{n}] \subset \mathbb{R}^2$ and let $k$ be an integer. Place points in $S_n$ according to a Poisson process of intensity $1$ and put an undirected edge between each point and its $k$ nearest neighbours. Let $S_{n,k}$ be the resulting random geometric graph.

Several authors (see below) have considered the following question:
for which $k$ is $S_{n,k}$ connected?  Of course, it is always
possible for $S_{n,k}$ to fail to be connected, no matter how large
$k$ is; the best we can hope for is that $S_{n,k}$ is connected
`asymptotically'. Formally, given a function $k\colon \N\to \N$ and a
property $\mathcal{Q}$ of geometric graphs, we say that $S_{n,k(n)}$
has a property $\mathcal{Q}$ \emph{with high probability} (abbreviated
to \whp) if
\[\lim_{n \rightarrow \infty} \mathbb{P}(S_{n, k(n)} \textrm{\ has property $\mathcal{Q}$}) =1.\]
We remark that what this is saying is that the probability a random point set gives rise to a graph with property $\mathcal{Q}$ tends to one.

Elementary arguments indicate that there exist constants $c_l$ and $c_u$ such that for every $c< c_l$, $S_{n, \lceil c \log n \rceil}$ is \whp not connected while for every $c>c_u$ $S_{n, \lfloor c \log n   \rfloor}$ is \whp connected. Using a result of Penrose~\cite{MR1442317}, Xue and Kumar~\cite{xue-kumar} showed that $c_u \leq 5.1774$. A bound of $c_u\leq 2\log \left(\frac{4\pi/3+\sqrt{3}/2}{\pi+3\sqrt{3}/4}\right)\approx 3.8597$ can also be read out of earlier work by Gonz\'ales-Barrios and Quiroz~\cite{MR1965368}.

These results were significantly improved by Balister, Bollob\'as, Sarkar and Walters in~\cite{MR2135151, MR2514943} where they established the existence of a critical constant $c_*: 0.3043 < c_* < 1/\log 7 \approx 0.5139$ such that for any
$c<c_*$ $S_{n,\lceil c \log \rceil}$ is \whp not connected and for any $c>c_*$ $S_{n, \lfloor c \log n \rfloor}$ is \whp connected.
They also made the following conjecture about the sharpness of the transition.

\begin{conjecture}[Conjecture 3 of~\cite{MR2135151}] \label{conjoc}
For any $0< \varepsilon <1$, there exists an integer constant $C(\varepsilon)$ such that for all $n$ sufficiently large,
if 
\[ \mathbb{P}(S_{n,k} \textrm{ is connected})\geq \varepsilon \]
then
\[\mathbb{P}( S_{n,k+C(\varepsilon)} \textrm{ is connected})> 1-\varepsilon.\]
\end{conjecture}

The main result of this paper is the following theorem which proves the conjecture for an explicit function $C(\eps)$.

\begin{theorem} \label{maintheo}
There exist absolute constants $C>0$ and $\gamma>0$ such that for every $0< \varepsilon < 1$ and all $n> \varepsilon^{-\gamma}$, if 
 \[\mathbb{P}(S_{n,k} \textrm{ is connected}) \geq \varepsilon\] 
then 
 \[\mathbb{P}(S_{n,k+ \lfloor C\log(1/\varepsilon)\rfloor} \textrm{ is connected}) > 1 -\varepsilon .\] 
\end{theorem}

In \cite{MR2479805} Balister, Bollob\'as, Sarker and Walters proved a weaker variant of their conjecture which they used to show that if $k=k(n)$ is such that $S_{n,k(n)}$ is connected \whp then for any $s=o(\log n)$ the graphs $S_{n,k'(n)}$ where $k'(n)=k(n)+ \lfloor 6 \sqrt{(s-1) \log n} \rfloor$ are \whp $s$-connected in a technical sense of `on average'. As an immediate corollary to Theorem \ref{maintheo}, we may remove the somewhat complicated hypothesis that they needed in the statement of their result: Theorem 10 of~\cite{MR2479805} (admittedly with a weaker constant). Moreover, in the final section we strengthen this substantially proving the following theorem.

\begin{theorem}\label{s-connected}
Whenever $k(n)$ is an integer sequence such that $S_{n,k(n)}$ is \whp
connected and $s(n)$ is an integer sequence with $s(n)=o(\log n)$, then
$S_{n, k(n)+ \lfloor 2C s\log \log n \rfloor}$ is \whp $s$-connected.
\end{theorem}
\noindent%
This proves the main conjecture in~\cite{MR2479805}.
\goodbreak

Before we describe the structure of our paper, we briefly contrast the $k$ nearest neighbours model with another classical random geometric graph model introduced by Gilbert~\cite{MR0132566}. As before, let $S_n$ be the square $[0,\sqrt{n}]\times [0,\sqrt{n}]\subset \mathbb{R}^2$. Let $r$ be a real number. Again, place points in $S_n$ according to a Poisson process of intensity $1$ but this time put an undirected edge between any pair of points which lie at a distance of at most $r$ from one another. We denote by $G_{n,r}$ the resulting random geometric graph model. $G_{n,r}$ is often known as the Gilbert disc model. Penrose~\cite{MR1442317,MR1704341,MR1986198} proved very precise results on the connectivity of $G_{n,r}$.  In particular he showed that isolated vertices are the main obstacle to connectivity in the sense that \whp
\[\inf\{r\geq 0: \ G_{n,r} \textrm{ is connected}\}=\inf\{r\geq 0: \ G_{n,r} \textrm{ has no isolated vertices}\}. \]

The situation is quite different for the $k$ nearest neighbours model,
which has no isolated vertices nor any immediately apparent analogous
family of geometric obstructions to connectivity --- indeed the
the value of the critical constant $c_*$ is not known (although it may
well be the lower bound of $0.3043\ldots$ proved in~\cite{MR2135151}).
 
One motivation for the study of $S_{n,k}$ (and the Gilbert disc model)
comes from the theory of ad-hoc wireless networks. We imagine that we
have various radio transmitters (nodes) that wish to communicate using
multiple hops. The transmitters could have fixed range which naturally
corresponds to the Gilbert disc model, or they could adjust their
power so that each node has some fixed number of neighbours which is
exactly the $k$-nearest neighbour model. In this context
Theorem~\ref{s-connected} is a result about the fault tolerance of
such a network: it says that we can have a fault tolerant network for
very little additional cost over the minimum needed for communication.
\subsection*{Outline of Paper}

In the first section, we adapt techniques first introduced in~\cite{MR2514943} to relate the global property of connectivity to certain families of local events: these will be events determined by  the Poisson process inside a square of area of order $\log n$. 

In the second section we prove a geometric lemma which is crucial to our argument, establishing that `small' connected components in $S_{n,k}$ have a region of `high point density'.

In the third section we show that removing points from such a dense region results in a much more likely configuration which still gives rise to a small connected component in the $k'$-nearest neighbour graph for some $k'$ a little smaller than $k$.  In other words the graph $S_{n,k'}$ is much more likely to be disconnected than $S_{n,k}$ which is exactly Theorem~\ref{maintheo}.

In the final section we prove Theorem~\ref{s-connected}.
\section{Local obstacles to connectivity}

Following~\cite{MR2514943}, we shall relate the global connectivity of $S_{n,k}$ to certain families of local events. Let $M$ be an integer constant which we shall specify later on. Let $U_n$ be the square
\[U_n= \left[\frac{-M\sqrt{\log n}}{2},\frac{M\sqrt{\log n}}{2}\right]\times \left[\frac{-M\sqrt{\log n}}{2},\frac{M\sqrt{\log n}}{2}\right] \subset \mathbb{R}^2.\] 
We shall refer to the subsquare $\frac{1}{2}U_n$ as the \emph{central subsquare} of $U_n$. Place points in $U_n$ according to a Poisson process of intensity $1$, and put an undirected edge between any point and the $k$ points nearest to it to obtain the random geometric graph $U_{n,k}$.

We define $A_k$ to be the event that $U_{n,k}$ has a connected component wholly contained inside the central subsquare $\frac{1}{2}U_n$. First, note that our $A_k$ event is slightly different from the family of events defined in~\cite{MR2514943}: there the size of the box corresponding to $U_n$ varied with $k$ rather than $\log n$. One of the advantages of our definition of $U_n$ is that the $A_k$-events are nested: if $k \leq k'$, then $A_{k'}\subseteq A_k$.  We shall cover most of $S_n$ with copies of $U_n$ and show (approximately) that $S_{n,k}$ is disconnected if and only if the event $A_k$ occurs in one of these copies.

For this argument to work we need to ensure that \whp  $S_{n,k}$ contains no `long' edges (relative to $M\sqrt{\log n}$) and only one connected component of `large' diameter. The following result is exactly what we want.

\begin{lemma}[Lemma~1 of \cite{MR2514943}]\label{mcon} 
For any fixed $\alpha_1, \alpha_2$ with $0<\alpha_1<\alpha_2$ and any $\beta>0$, there exists $c=c(\alpha_1,\alpha_2,\beta)>0$, depending only on $\alpha_1,\alpha_2$ and $\beta$, such that for any $k$ with $\alpha_1\log n\leq k \leq \alpha_2 \log n$, the probability that $S_{n,k}$ contains two components each of diameter at least $c\sqrt{\log n}$ or any edge of length at least $c\sqrt{\log n}$ is $O(n^{-\beta})$.
\end{lemma}
{\it Remark:} In this paper we use the $O$ notation in a slightly non-standard way. Most of our results depend on $n$ and $k$ where $k=k(n)$ is a function of $n$. When we say $f(n,k)=O(n)$ we mean `uniformly in $k$': that is there is a constant $B$ such that $f(n,k)\le B n$ for all $n$ and $k$ (satisfying our other constraints).

Let $M= \max \left(\lceil 16 c(0.3,0.6,2)\rceil , 30 \right)$. In our argument we shall also need the following lemma, which is an easy modification of Corollary~6~of~\cite{MR2514943}.

\begin{lemma}\label{no long edges in U}
For any $n$ and any integer $k$ with $0.3 \log n < k <0.6 \log n$, the probability that $U_{n,k}$ contains an edge of length at least $\frac{M\sqrt{\log n}}{8}$ is $O(n^{-6})$.
\end{lemma}

\begin{proof}

This is very similar to the proof of Corollary~6 of~\cite{MR2514943}, but we have to make allowances for the slight  difference in our definition of the event $A_k$. 

Let $k <0.6 \log n$. Suppose some vertex $x\in U_n$ has its $k^{\textrm{th}}$ nearest neighbour lying at a distance of at least $\frac{M\sqrt{\log n}}{8}$. Then there must be fewer than $k<0.6\log n$ points within a quarter-disc about $x$ of area $\frac{\pi M^2 \log n}{256}$. (We need to consider quarter-discs since $x$ may be close to a corner of $U_n$.) Since we picked $M\geq 30$, we have $\frac{\pi M^2 \log n}{256}> 10\log n$. Let $X \sim \textrm{Poisson}(10\log n)$. Then,
\begin{align*}
\mathbb{P}(X<0.6 \log n) &= \sum_{s< 0.6 \log n} \frac{{(10 \log n)}^s }{s!} e^{-10\log n} \\
& < \left(0.6 \log n\right) {\left(\frac{10\log n}{0.6 \log n /e}\right)}^{0.6 \log n}e^{-10\log n}\\
& < 0.6 (\log n) e^{(0.6 \log (50e/3) -10)\log n}\\&< e^{-7 \log n} \qquad\textrm{for $n$ sufficiently large.} 
\end{align*}
Thus the probability that any vertex $x \in U_n$ has its $k^{\textrm{th}}$ nearest neighbour lying at distance at least $\frac{M\sqrt{\log n}}{8}$ away is at most
\begin{align*}
\mathbb{E}\{\textrm{number of vertices in }U_n\} \ \times\  \mathbb{P}(X<0.6 \log n) 
&< M^2 (\log n) e^{-7 \log n} \\&=O\left(n^{-6}\right),
\end{align*}
as required.
\end{proof}

We also need to define what we meant by `most' of $S_n$. Let 
\[
T_n=\left[M\sqrt{\log n},\left(\left\lfloor\tfrac{\sqrt{n}}{M\sqrt{\log n}}\right \rfloor -1\right) M\sqrt{\log n}\right]^2.
\]
The nice feature of $T_n$ is that it is not very close to any of the boundary of $S_n$. The following lemma is a minor restatement of Theorem~1 of \cite{2011arXiv1101.2619W}.
\begin{lemma}\label{l:MJW}
There is a positive constant $0<c_1 <2$ such that if $k>0.3\log n$ then the probability
that $S_{n,k}$ contains any component of diameter $O(\sqrt{\log n})$
not wholly contained in $T_n$ is $O(n^{-c_1})$.
\end{lemma}
We now define two covers of $T_n$ by copies of $U_n$. The \emph{independent} cover $\mathcal{C}_1$ of $T_n$ is obtained by covering $T_n$ with copies of $U_n$ with disjoint interiors. The \emph{dominating} cover $\mathcal{C}_2$ of $T_n$ is obtained from $\mathcal{C}_1$ by replacing each square $V\in\mathcal{C}_1$ by the sixteen translates $V+(i \frac{M\sqrt{\log n}}{4}, j \frac{M\sqrt{\log n}}{4})$, $i,j\in\{0,1,2,3\}$. By construction, we have $\mathcal{C}_1 \subseteq \mathcal{C}_2$ and the copies of $\frac{1}{4}U_n$ corresponding to elements of $\mathcal{C}_2$ cover the whole of $T_n$. Also $|\mathcal{C}_2|< 16 \frac{n}{M^2 \log n}$.

We shall write `$A_k$ occurs in $\mathcal{C}_i$' as a convenient shorthand for  `there is a copy $V$ of $U_n$ in $\mathcal{C}_i$ for which the event corresponding to $A_k$ occurs'. We shall also write $V_k$ for the $k$-nearest neighbour graph on $V$, and $\frac12 V$ for the centre subsquare of $V$.

Lemmas~\ref{mcon}~and~\ref{no long edges in U} allow us to relate, up to some small error, the global connectivity to the local events $A_k$. Before we make this relationship precise we need a technical lemma.

\begin{lemma}\label{locality}
Suppose $S_{n,k}$ contains no edge of length greater than $\frac{M   \sqrt{\log n}}{16}$ and that $V\in \mathcal{C}_2$ is a copy of $U_n$ such that $V_k$ contains no edge of length greater than $\frac{M\sqrt{\log n}}{8}$. Then $S_{n,k}$ has a connected component contained inside $\frac{1}{2}V$ whenever the event corresponding to $A_k$ occurs in $V$.
\end{lemma}
\begin{proof}Let $\Gamma_V$ denote the subgraph of $V_k$ consisting of all edges with at least one end in $\frac12 V$, and let $\Gamma_S$ be the subgraph of $S_{n,k}$ consisting of all edges with at least one end in $\frac 12 V$. We aim to show that $\Gamma_V=\Gamma_S$. Obviously this will imply the lemma.

Trivially, $S_{n,k}[V]$ is a subset of $V_k$. What extra edges can there be in $V_k$? We are assuming that $S_{n,k}$ contains no edges of length greater than $\frac{M\sqrt{\log n}}{16}$. Thus only the vertices within distance $\frac{M\sqrt{\log n}}{16}$ of the boundary of $V$ may be joined in $S_{n,k}$ to points in $S_n \setminus V$. So every edge in $V_k \setminus S_{n,k}[V]$ (i.e., all extra edges) must meet one of these vertices.%

Now $V_k$ contains no edges of length greater than $\frac{M\sqrt{\log n}}{8}$ , so that all the vertices meeting an edge of $V_k\setminus S_{n,k}[V]$ must lie a distance at most
\[
\frac{M\sqrt{\log n}}{8}+ \frac{M\sqrt{\log n}}{16} < \frac{M \sqrt{\log n}}{4}
\]
from the boundary of $V$.  Since the vertices inside the central subsquare $\frac{1}{2}V$ all lie at distance at least $\frac{M\sqrt{\log n}}{4}$ from the boundary of $V$, they do not meet any extra edges, and we have $\Gamma_V=\Gamma_S$ as claimed.
\end{proof}

\begin{theorem}\label{approx}
For all $n \in \mathbb{N}$ and all integers $k$ with $ 0.3 \log n< k< 0.6 \log n$, and $c_1$ as given by Lemma~\ref{l:MJW},
\[\mathbb{P}(S_{n,k} \textrm{ not connected}) = \mathbb{P}(A_k \textrm{ occurs in } \mathcal{C}_2)+O(n^{-c_1}). \]  
\end{theorem}
\begin{proof}
Suppose that $A_k$ occurs in $\mathcal{C}_2$. Then there is a copy $V$ of $U_n$ in $\mathcal{C}_2$ for which $A_k$ occurs; in other words, $V_k$ has a connected component $X$ wholly contained inside the central subsquare $\frac{1}{2}V$. By Lemma~\ref{mcon} and our choice of $M$, the probability that $S_{n,k}$ contains an edge of length at least $\frac{M\sqrt{\log n}}{16}$ is $O(n^{-2})$. Let us assume this does not happen. Then there are no edges between $\frac{1}{2}V$ and $S_n \setminus V$ in $S_{n,k}$. It follows that $X$ is a connected component in $S_{n,k}$ as well as in $V_k$, so that $S_{n,k}$ is disconnected. Thus 
\[\mathbb{P}(S_{n,k} \textrm{ not connected}) \geq \mathbb{P}(A_k \textrm{ occurs in } \mathcal{C}_2)+O(n^{-2}).\]

Conversely, suppose $S_{n,k}$ is not connected.
It must contain at least two connected components. By Lemma~\ref{mcon} and our choice of $M$, the probability that $T_{n,k}$ contains any edge of length at least $\frac{M\sqrt{\log n}}{16}$ or two components of diameter at least $\frac{M\sqrt{\log n}}{16}$ is at most $O(n^{-2})$. By Lemma~\ref{l:MJW} the probability that there is a small component not contained entirely within $T_n$ is $O(n^{-c_1})$. 
Also by Lemma \ref{no long edges in U}, the probability that $U_{n,k}$ has any edge longer than $\frac{M\sqrt{\log n}}{8}$ is $O(n^{-6})$. The probability that $V_k$ has an edge longer than $\frac{M\sqrt{\log     n}}{8}$ for some copy $V$ of $U_n$ in $\mathcal{C}_2$ is therefore at most $|\mathcal{C}_2| O(n^{-6}) = O(n^{-5})$. 
Thus the probability of any of the above occuring in $S_{n,k}$ is at most $O\left ( n^{-c_1}\right)$. 

From now on let us assume none of the above occur.
Then at least one of the connected components of $S_{n,k}$ is contained in $T_n$ and has diameter less than $\frac{M\sqrt{\log     n}}{16}$. Let $X$ be such a component and $x$ be a vertex of $X$. By our definition of $\mathcal{C}_2$ there is a copy $V$ of $U_n$ such that $x \in \frac{1}{4}V$. For any point $x' \notin \frac{1}{2}V$, we have $d(x,x')> \frac{M\sqrt{\log n}}{8}$. By our assumption on the diameter of $X$, we have that $x'\notin X$ and hence $X \subseteq \frac{1}{2}V$. So $X$ is contained entirely inside the central subsquare $\frac{1}{2}V$. Now $V_k$ and $S_{n,k}$ satisfy the hypotheses of Lemma~\ref{locality}, hence the event corresponding to $A_k$ occurs in $V$, and
\begin{align*}
\mathbb{P}(S_{n,k} \textrm{ not connected}) 
\leq \mathbb{P}(A_k \textrm{ occurs in } \mathcal{C}_2)+O(n^{-c_1}).
\end{align*}
The theorem follows. 
\end{proof}

Roughly speaking $\mathbb{P}(A_k \textrm{ occurs in }C_2)$ is of order $\frac{n}{\log n}\mathbb{P}(A_k)$ so, from a heuristic perspective, Theorem~\ref{approx} tells us that as we increase $k$ the transition of $S_{n,k}$ from \whp  not connected to \whp  connected happens at the same time as the transition from $\mathbb{P}(A_k)\gg \frac{\log   n}{n}$ to $\mathbb{P}(A_k) \ll \frac{\log n}{n}$.
The following is a precise statement of this relationship.

\begin{corollary}\label{appimp}
There exists a constant $c_2>0$ such that for all $\varepsilon:\ 0< \varepsilon \leq \frac{1}{2}$, all integers $n> \varepsilon^{-c_2}$ and all integers $k:\ 0.3\log n< k < 0.6 \log n$, if
\[\mathbb{P}(S_{n,k} \textrm{ connected})\geq \varepsilon\]
holds then 
\[\mathbb{P}(A_k) \leq e\log\left(\frac{1}{\varepsilon}\right)\frac{M^2\log n}{n}.\]
Conversely, if
\[ \mathbb{P}(A_k) \leq \frac{\varepsilon}{e^4} \frac{M^2\log n}{n},\]
then
\[\mathbb{P}(S_{n,k} \textrm{ connected}) > 1-\varepsilon.\] 
\end{corollary}

\noindent%
{\it Remark:\ } There is nothing special about the constants $e$ and $e^4$\/: we picked these values for later convenience, but all we needed was $e>2$ and $e^4>16$.

\begin{proof}
Suppose $\mathbb{P}(S_{n,k}\textrm{ is connected})\geq \varepsilon$.
The copies of $U_n$ contained in $\mathcal{C}_1$ have disjoint interiors, hence the event corresponding to $A_k$ occurs in each of them independently. Therefore 
\[\mathbb{P}(A_k \textrm{ occurs in } \mathcal{C}_1)= 1-{(1-\mathbb{P}(A_k))}^{|\mathcal{C}_1|}.\]
Now,
\begin{align*}
\mathbb{P}(A_k \textrm{ occurs in } \mathcal{C}_1)&\leq \mathbb{P}(A_k \textrm{ occurs in } \mathcal{C}_2)
\qquad&\text{since $\mathcal{C}_1\subset \mathcal{C}_2$}\\
&= \mathbb{P}(S_{n,k} \textrm{ not connected})+O(n^{-c_1})\qquad&\text{by Theorem~\ref{approx}}\\
& \leq 1- \varepsilon +O(n^{-c_1}).
\end{align*}
Hence,
\[{(1-\mathbb{P}(A_k))}^{|\mathcal{C}_1|} \geq \varepsilon +O(n^{-c_1}).  \]
Hence, provided we chose $c_2$ large enough, we see that, for all $n>\varepsilon^{-c_2}$, the right hand side is at least $\frac{\varepsilon}{2}$. Taking logarithms on both sides and using the inequality $\log(1-x)\leq -x$ for $0\leq x\leq 1$ yields
\[-|\mathcal{C}_1|\mathbb{P}(A_k)\geq \log (\varepsilon/2)  \]
so
\[\mathbb{P}(A_k) \leq \frac{1}{|\mathcal{C}_1|}\left(\log \frac{1}{\varepsilon/2}\right)= \frac{1}{|\mathcal{C}_1|}\left(\log \frac{1}{\varepsilon}+\log 2\right). \]
Now $\mathcal{C}_1$ contains $\frac{n}{M^2\log n}(1+O(\sqrt{\frac{\log     n}{n}}))$ copies of $U_n$, $0<\varepsilon \leq \frac{1}{2}$ and $e>2$. Hence, provided that we choose our constant $c_2$ sufficiently large, for all $n>\varepsilon^{-c_2}$ we have
\[\mathbb{P}(A_k) \leq \frac{eM^2\log n}{n} \log \frac{1}{\varepsilon}.\]

For the converse
suppose that $\mathbb{P}(A_k) \leq \varepsilon\frac{M^2\log   n}{e^4n}$. By Theorem \ref{approx} we have
\begin{align*}
\mathbb{P}(S_{n,k} \textrm{ not connected})&=
\mathbb{P}(A_k \textrm{ occurs in } \mathcal{C}_2)+O(n^{-c_1})\\
&\leq |\mathcal{C}_2| \mathbb{P}(A_k)+O(n^{-c_1})\\
&\leq |\mathcal{C}_2|\varepsilon\frac{M^2\log n}{e^4n} +O(n^{-c_1})\\
&\leq \varepsilon\frac{16}{e^4}+O(n^{-c_1})\qquad\qquad&\text{since $|\mathcal{C}_2|<\tfrac{16 n}{M^2\log n}$.}
\end{align*}

Since $0< \varepsilon \leq \frac{1}{2}$ and $\frac{16}{e^4}<1$, we have (again providing we chose $c_2$ sufficiently large) for all $n> \varepsilon^{-c_2}$,
\[\mathbb{P}(S_{n,k} \textrm{ not connected}) < \varepsilon. \]
\end{proof}

\section{Small components have high point density}
Having made precise the relationship between $\mathbb{P}(A_k)$ and $\mathbb{P}(S_{n,k} \textrm{ connected})$, we turn our attention to $A_k$. Our aim in this section is to show that provided $k>0.3\log n$, small connected components in $U_{n,k}$ witnessing  $A_k$ must have a region with `high point density'.

Let $N$ be an integer constant whose value we shall specify later. We consider a perfect tiling of $U_n$ by square tiles of area $\frac{\log n}{N^2}$. (Such a perfect tiling exists as $U_n$ has area $M^2\log n$ and $M$, $N$ are integers.) The expected number of points of the Poisson point process on $U_n$ in each tile is $\frac{\log   n}{N^2}$. Fix $0<\eta\leq\frac{1}{2}$. Given a tile $Q$, we say that the event $A_{k,Q}$ occurs if $A_k$ occurs \emph{and} the tile $Q$ receives more than $(1+\eta)\frac{\log n}{N^2}$ points. Similarly, we say that the event $A_{k,Q}'$ occurs if $A_k$ occurs \emph{and} the tile $Q$ receives more than $(1+\frac{\eta}{2})\frac{\log n}{N^2}$ points.

\begin{lemma}\label{geomlem}
Suppose $k\in[0.3 \log n,0.6\log n]$. Then
\[\mathbb{P}(A_k\setminus \bigcup_Q A_{k,Q}) =O(n^{-1.1}).\]
\end{lemma}

 The main idea of the proof of this geometric lemma is the following: suppose $X$ is a connected component of $U_{n,k}$ wholly contained inside $\frac{1}{2}U_n$, and suppose $x$ is a vertex of $X$ which lies `on the boundary' of $X$. Write $r$ for the distance between $x$ and its $k$-th nearest neighbour. 

If $U_{n,k}$ contains no tile with high density (i.e. no tile receiving more than $(1+\eta)$ times the expected number of points), then then intersection of the ball of radius $r$ centred at $x$ with the `convex hull' of $X$ must have large area (about $\frac{k}{1+\eta}-o(k)$). In particular looking outwards from $X$ at $x$ there must be quite a few empty tiles. Doing the above in several different directions one gets that $X$ is surrounded by a wide `sea' of empty tiles of area at least $1.1\log n$. Since the number of tiles $M^2N^2$ is a constant, the probability that such a collection of empty tiles exists is
$O(n^{-1.1})$, yielding the desired result.

Before we start, we need the following technical result.

\begin{lemma}\label{boundo}
Let $\gamma: [0,1]\rightarrow U_n$ be a closed continuously differentiable curve in $U_n$. Let $l(\Gamma)$
be the length of the curve $\Gamma=\gamma([0,1])$, and let $D$ be the number of tiles it meets. Then
\[D \leq \frac{9 l(\Gamma)}{\sqrt{\log n}/N}.\]
\end{lemma}
\begin{proof}
We define a graph $G$ on the set of tiles of $U_n$ by setting an edge between tiles $Q$ and $Q'$ if they meet in at least one point. ($G$ is just the usual square integer lattice on ${\{1,2, \ldots MN\}}^2$ with diagonal edges added.) Every tile has at most $8$ neighbours in this graph. Let $S$ be the set of tiles met by $\Gamma$. Greedily pick a maximal subset $S'\subseteq S$ which is independent in $G$: pick the tile $Q_1$ with $\gamma(0)\in Q_1$, then pick the first nonadjacent tile $Q_2$ which $\gamma(t)$ next meets and so on.  We have $D=|S|\leq 9|S'|$. Now $\Gamma$ is continuous and cycles through the tiles of $S'$ before coming back to $Q_1$. Since the minimum distance between points lying in nonadjacent tiles is  at least one tile length (i.e., $\frac{\sqrt{\log n}}{N}$), it follows that the length of $\Gamma$ satisfies
\[l(\Gamma) \geq |S'|\frac{\sqrt{\log n}}{N}.\]  
Substituting $D\leq 9|S'|$ and rearranging terms, we get the desired inequality
\[D \leq \frac{9l(\Gamma)}{\sqrt{\log n}/N}. \ \]
\end{proof}

\begin{proof}[Proof of Lemma~\ref{geomlem}]
Let $k$ be an integer with $ 0.3 \log n<k <0.6\log n$. By Lemma~\ref{no long edges in U} the probability of $U_{n,k}$ containing any edge of length at least $\frac{M\sqrt{\log n}}{8}$ is $O(n^{-6})$. Since we are trying to show $A_k \setminus \bigcup_Q A_{k,Q}$ has probability at most $O(n^{-1.1})$, we may assume in what follows that all edges in $U_{n,k}$ have length strictly less than $\frac{M\sqrt{\log n}}{8}$.

Suppose $\mathcal{P}$ is a pointset for which $A_k$ occurs but $A_{k,Q}$ does not occur for any tile $Q$. Write $U_{n,k}(\mathcal{P})$ for the $k$ nearest neighbours graph on $U_n$ associated with the pointset $\mathcal{P}$. Let $X$ be the set of vertices of a connected component of $U_{n,k}(\mathcal{P})$ wholly contained in $\frac{1}{2}U_n$. Using an idea of Balister, Bollob\'as, Sarkar and Walters~\cite{MR2135151}, we shall consider the \emph{hexagonal hull} of $X$, $H(X)$, which we now define. 

\begin{figure}
\begin{center}
\includegraphics{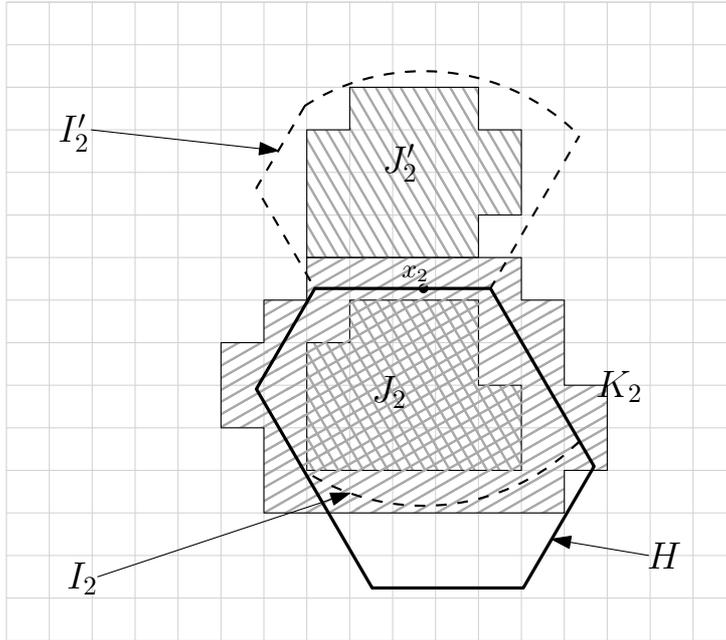}
\end{center}
\caption{The hexagonal hull $H$ and regions  $I_2,J_2,K_2$ and $I_2',J_2'$.}
\end{figure}

We consider the six tangents to the convex hull of $X$ making angles of $0$, $\frac{\pi}{3}$ and $\frac{2 \pi}{3}$ with the $x$-axis (two for each angle). Together, these define a hexagon $H(X)$ containing $X$ whose edges are segments of the tangents (some of which may have zero length). We shall call $H(X)$ the hexagonal hull of $X$, and label its edges $E_1, E_2, \ldots E_6$ in cyclic order so that the top and bottom edges parallel to the $x$-axis are $E_2$ and $E_5$ respectively .

Consider $E_1$. There exists $x_1 \in E_1 \cap X$. Let $r_1$ be distance between $x_1$ and its $k$-th nearest neighbour. Let $I_1$ be the intersection of the ball of radius $r_1$ centred at $x_1$ with the hexagon $H(X)$. Let $I_1'$ be the reflection of $I_1$ with respect to $L_1$. Since $I_1' \subset U_n \setminus H(X)$ and since every point of $I_1'$ lies at distance at most $r_1$ from $x_1$, it follows that $I_1'$ contains no point of $\mathcal{P}$. We shall show that $I_1'$ covers many tiles.

Let $J_1$ be the union of all of the tiles wholly contained inside $I_1$, and let $J_1'$ be the union of all of the tiles wholly contained inside $I_1'$. Let $K_1$ be the union of all of the tiles meeting $I_1$ and let $K_1'$ be the union of all of the tiles meeting $I_1'$. Since no tile in $U_n$ contains more than $(1+\eta)\frac{\log n}{N^2}$ points of $\mathcal{P}$, it follows that $K_1$ is the union of at least $\frac{k}{(1+\eta)\log n/N^2}$ tiles.

A tile is contained in $K_1 \setminus J_1$ only if it meets the boundary of $I_1$. Now, since $I_1$ is a convex subset of a disc of radius $r_1$, the boundary of $I_1$ has length less than $2\pi r_1$, so by Lemma \ref{boundo}, $K_1 \setminus J_1$ is the union of at most $\frac{18 \pi r_1}{\sqrt{\log n}/N}$ tiles. By the same argument, $K_1'\setminus J_1'$ is the union of at most $\frac{18 \pi r_1}{\sqrt{\log n}/N}$ tiles. Denote by $|I_1|$ the area of $I_1$, and similary $|I_1'|, |J_1|, \ldots |K_1'|$. We have
\begin{align*}
|J_1'| & \geq |I_1'|-|K_1'\setminus J_1'| \\
&\geq |I_1|-|K_1'\setminus J_1'|\\
&\geq |K_1|-|K_1\setminus J_1|-|K_1'\setminus J_1'|. 
\end{align*}
Now each tile has area $\frac{\log n}{N^2}$. We therefore have
\begin{align*}
|J_1'| &\geq \frac{\log n}{N^2} \left(\frac{k}{(1+\eta)\log n /N^2}- \frac{36 \pi r_1}{\sqrt{\log n}/N} \right)\\
& \geq \frac{k}{1+\eta} -\frac{36 \pi r_1 \sqrt{\log n}}{N} \\
& \geq  \frac{k}{1+\eta} -\frac{9 M \pi \log n }{2 N} \qquad\textrm{since }r_1< \frac{M\sqrt{\log n}}{8}.
\end{align*}
We turn at last to the choice of $N$: let $N= 10 \lceil 27 M \pi \rceil$. For $k>0.3 \log n$ and $\eta \leq \frac{1}{2}$, the above becomes:
\begin{align*}
|J_1'| & > \frac{11}{60}\log n.
\end{align*}

For $i=2,3\ldots 6$ we may define $I_i$, $I_i'$, etc... as above. It is easy to see that the $J_i'$ are disjoint: each $J_i'$ lies between the bisectors of two adjacent angles of the convex hexagon $H(X)$. Repeating the argument above to bound below $|J_2'|$, $\ldots \ |J_6'|$, we get:
\begin{align*}
\left|\bigcup_{i=1}^{6} J_i'\right|&=\sum_{i=1}^{6} |J_i'|\\
&> \frac{11}{10}\log n.
\end{align*}
Thus there are at least $\frac{11}{10}\log n/(\log n / N^2)= 110 {(\lceil 27 M \pi \rceil)}^2$ tiles which receive no points. There are at most $\binom{M^2N^2}{110{(\lceil 27 M \pi \rceil)}^2}$ ways of choosing this many tiles. Since $M$ and $N$ are constants this is just a (large) constant.
The probability that there exist $110 {(\lceil 27 M \pi   \rceil)}^2$ empty tiles (i.e., empty tiles with total area $\frac{11}{10}\log n$) in $U_n$ is therefore
\[O( \exp (-\frac{11}{10}\log n)) =O(n^{-1.1}).\] 
Thus
\[\mathbb{P}(A_k \setminus \bigcup_Q A_{k,Q}) = O(n^{-1.1}),\]
as claimed.

\end{proof}

\section{The sharp connectivity threshold for $S_{n,k}$}

In Lemma~\ref{geomlem} of the previous section we proved that small components witnessing $A_k$ have high point density. We use this fact to prove a sharpness result for $\mathbb{P}(A_k)$, which by Corollary~\ref{appimp} implies in turn a sharp threshold for the connectivity of $S_{n,k}$ (i.e., Theorem~\ref{maintheo}). We shall do this by showing that, for all $k'>k$, most pointsets in $A_{k'}$ may be obtained by adding points to already dense parts of $A_k$ pointsets.

We shall need the following lemma, which is a convenient restatement of Theorem~5 of~\cite{MR2135151}.

\begin{lemma}\label{prep}
There exists a positive constant $c_3>0$ such that for every $\varepsilon$ with $0 < \varepsilon \leq \frac{1}{2}$ and all $n>\varepsilon^{-c_3}$, 
\begin{align*}
\textrm{if $k \leq 0.3 \log n$, then }& \mathbb{P}(S_{n,k} \textrm{ connected}) < \varepsilon, \\
\textrm{and if $k \geq 0.6 \log n$, then }& \mathbb{P}(S_{n,k} \textrm{ connected}) > 1-\varepsilon. \\
\end{align*}
\end{lemma}

Recall that in the previous section we fixed constants $0< \eta \leq \frac{1}{2}$ and $N\in \mathbb{N}$ and introduced a tiling of $U_n$ into $M^2N^2$ small square tiles as well as the families of events $A_{k,Q}$ and $A_{k,Q}'$. Lemma~\ref{geomlem} says that provided $\mathbb{P}(A_k) = \Omega(n^{-1})$, we have
$\mathbb{P}(\bigcup_Q A_{k,Q}) =(1-O(n^{-0.1}))\mathbb{P}(A_k)$. 
Thus if a small $A_k$ connected component occurs, then with high probability some tile $Q$ receives far more points than expected. We show that if $k'>k$ then most $A_{k'}$ pointsets can be  obtained by adding points to an overpopulated tile of an $A_k$ pointset.

We need one more piece of notation: given a tile $Q$ let 
 $A_{k, Q, L}$ be the event that if we remove any $L$ points from $Q$ then $A_{k,Q}'$ occurs.
\begin{lemma}\label{remove}
For any tile $Q$ and positive integer $L< \frac{\eta \log n}{2N^2}$  we have
\[A_{k+L,Q} \subseteq A_{k,Q, L}.\]
\end{lemma}
\begin{proof}
Suppose that  $\mathcal{P}\subset U_n$ is a pointset for which the event $A_{k+L,Q}$ occurs. It is enough to show that the removal of any $L$ points from $\mathcal{P}\cap Q$ yields a pointset $\mathcal{P}'$ for which the event $A_{k,Q}'$ occurs.

As in Lemma~\ref{geomlem}, write $U_{n,k}(\mathcal{P})$ for the $k$ nearest neighbours graph on $U_n$ associated with the pointset $\mathcal{P}$. Since we remove at most $L$ vertices from $\mathcal{P}$ every vertex in $\mathcal{P}$ loses at most $L$ of its $k+L$ nearest neighbours; the set of its $k$ nearest neighbours in $\mathcal{P}'$ is thus a subset of the set of its $k+L$ nearest neighbours in $\mathcal{P}$. It follows that $U_{n,k}(\mathcal{P}')$ is a subgraph of $U_{n,k+L}(\mathcal{P})$.

$U_{n,k+L}(\mathcal{P})$ has a connected component wholly contained inside $\frac{1}{2}U_n$. This component must contain at least $k+L+1>L$ vertices and since we have removed only $L$ vertices from $\mathcal{P}$ to obtain $\mathcal{P}'$ some vertices of this component remain: that is, $U_{n,k}(\mathcal{P}')$ must also have component wholly contained inside $\frac{1}{2}U_n$.
Thus $\mathcal{P}'\in A_k$. 

Moreover the number of points in $\mathcal{P}'\cap Q$ is exactly
\begin{align*}
|\mathcal{P} \cap Q| -L &> (1+\eta)\frac{\log n}{N^2} -\frac{\eta \log n}{2 N^2}
= (1+\frac{\eta}{2})\frac{\log n}{N^2}
\end{align*} 
and hence $\mathcal{P}'\in A_{k,Q}'$ as claimed.
\end{proof}

\begin{corollary}\label{corolb}
Let $L< \frac{\eta \log n}{2N^2}$ be a positive integer and $Q$ a tile. Then
\[\mathbb{P}(A_{k+L,Q}) < (1+\frac{\eta}{2})^{-L} \mathbb{P}(A_{k,Q}').\] 
\end{corollary}
\begin{proof}
First, note that we may consider the Poisson process on $U_n$ as the union of a Poisson process on $Q$ and an independent Poisson process on the disjoint set $U_n\setminus Q$. Now a Poisson point process on $Q$ is just a uniform point process placing \[Z\sim~\textrm{Poisson}\left(\frac{\log n}{N^2}\right)\] points in $Q$. 
\goodbreak

We may think of this uniform point process as adding points one by one. If $A_{k,Q,L}$ occurs then in particular $A_{k,Q}'$ occurs if we remove the last $L$ points added by the point process. It follows that
\begin{align*}
\mathbb{P}(A_{k+L,Q}) & \leq \mathbb{P}(A_{k,Q,L}) &\textrm{by Lemma~\ref{remove}} \\
&= \sum_m \mathbb{P}(A_{k,Q,L}|Z=m+L) \mathbb{P}(Z=m+L) \\
& \leq \sum_m\mathbb{P}(A_{k,Q}'|Z=m) \mathbb{P}(Z=m+L) &\textrm{by definition of $A_{k,Q,L}$}\\
&= \sum_m \mathbb{P}(A_{k,Q}'|Z=m) \mathbb{P}(Z=m) \prod_{i=1}^L \frac{N^{-2} \log n }{m+i}
\end{align*}
By the definition of $A_{k,Q}'$, 
\[\mathbb{P}(A_{k,Q}'|Z=m)=0 \qquad\textrm{ for all }m< (1+\frac{\eta}{2})\frac{\log n}{N^2}.\] 
For $m\geq (1+\frac{\eta}{2})\frac{\log n}{N^2}$, we have
\[
\prod_{i=1}^L \frac{N^{-2}\log n }{m+i}  < 
\left(\frac{N^{-2}\log n}{m}\right)^L
\le {(1+\frac{\eta}{2})}^{-L}.
\]
It follows that
\[\mathbb{P}(A_{k+L,Q})< {(1+\frac{\eta}{2})}^{-L} \mathbb{P}(A_{k,Q}'),\]
as claimed.
\end{proof}

\begin{theorem}\label{aksha}
There are constants $c_4$ and $L \in \mathbb{N}$ such that
for all $n> c_4$ and all $k$ with
\[k\in[0.3\log n, 0.6 \log n] \textrm{ and } \mathbb{P}(A_k) \geq n^{-1.05} \]
we have:
\[\mathbb{P}(A_{k+L}) < e^{-1}\,   \mathbb{P}(A_{k}).\]

\end{theorem}
\begin{proof}
Let $L$ be an integer constant which we shall specify later on. As $\eta$, $L$ and $N$ are all constants, for an appropriate choice of our constant $c_4>0$ and all $n>c_4$, we have $L < \frac{\eta \log n}{2N^2}$ so that the hypothesis of Corollary~\ref{corolb} is satisfied. Also $k\in[0.3 \log n, 0.6\log n]$, so the hypothesis of Lemma~\ref{geomlem} is satisfied as well. Applying the two lemmas successively, we get:
\begin{align*}
\mathbb{P}(A_{k+L}) &= \mathbb{P}(\bigcup_Q A_{k+L,Q})+O(n^{-1.1}) \qquad&\text{by Lemma~\ref{geomlem}}\\ 
& \leq \sum_Q \mathbb{P}(A_{k+L,Q})+O(n^{-1.1}) \\
& \leq \sum_Q (1+\frac{\eta}{2})^{-L} \mathbb{P}(A_{k,Q}')+O(n^{-1.1}) &\text{by Corollary~\ref{corolb}}\\
& \leq M^2N^2 (1+\frac{\eta}{2})^{-L} \mathbb{P}(A_k)+O(n^{-1.1})&\text{(since $\Prb(A_{k,Q}')\le \Prb(A_k)$.)}
\end{align*}
We now choose  $L$: let
\[L =\left\lceil  \frac{\log\left( M^2N^2e^2\right)}{\log (1+\frac{\eta}{2})}\right\rceil\] 
so that
\[M^2N^2 (1+\frac{\eta}{2})^{-L} \leq e^{-2} .\]
By assumption $\mathbb{P}(A_k)\geq n^{-1.05}$, so for an appropriate choice of our constant $c_4>0$ and all $n>c_4$, we have
\[\mathbb{P}(A_{k+L}) < e^{-1}\, \mathbb{P}(A_k),\]
as claimed. (Note that the choice of our constant $L$ depended only on the constant $M$, $N$ and $\eta$.) 
\end{proof}

\begin{proof}[Proof of Theorem~\ref{maintheo}]
In essence, we just iterate Theorem~\ref{aksha}. However, we have to choose the right parameters and make sure the conditions hold at each stage.

We choose $\gamma>0$ such that $\gamma>\max\left(c_2, c_3, \log_2
c_4,20 \right)$.  Note that, since $M\geq 30$, we have
$\frac{e^4}{M^2\log n}\leq \frac{e^4}{900 \log 2}<0.09 $ for all
$n\geq 2$ so $n^{-1/\gamma}>\frac{e^4}{M^2\log n}n^{-0.05}$ for all
$n\ge2$.

Suppose that $n$ and $k$ are such that $\Prb(S_{n,k} \textrm{ is
  connected})>\eps$ and $n>\eps^{-\gamma}$.  We may assume that
$\varepsilon \leq \frac{1}{2}$ and $\mathbb{P}(S_{n,k} \textrm{
  connected})\leq 1-\varepsilon$, for otherwise we have nothing to
prove.  Since $n>\eps^{-\gamma}>\eps^{-c_3}$ and
$\varepsilon<\mathbb{P}(S_{n,k} \textrm{ connected})<1-\eps$,
Lemma~\ref{prep} implies that $0.3 \log n<k<0.6\log n$. Thus, for $n>
\varepsilon^{-\gamma}$, the assumptions of Corollary~\ref{appimp} and
Theorem~\ref{aksha} are therefore satisfied.

Let $C$ be a strictly positive real constant which we shall specify
later on. There are three cases to consider.

Suppose first of all that
\[k+\lfloor C\log \frac{1}{\varepsilon} \rfloor\geq 0.6 \log n.\]
Then by Lemma~\ref{prep} we have $\mathbb{P}(S_{n,k+\lfloor C\log(1/\eps)\rfloor} \textrm{
  connected}) > 1- \varepsilon$, and we are done.

Secondly suppose  that $k+\lfloor C\log \frac{1}{\varepsilon} \rfloor< 0.6 \log n$ and
\[\mathbb{P}(A_{k+\lfloor C\log (1/\varepsilon) \rfloor})<n^{-1.05}.\] 
Since $n>\eps^{-\gamma}$, \[n^{-1.05} <n^{-1/\gamma}\frac{M^2\log n}{e^4n}< \varepsilon \frac{M^2\log n}{e^4 n},\] 
so that by Corollary \ref{appimp} we have $\mathbb{P}(S_{n,k+\lfloor C\log(1/\eps)\rfloor} \textrm{ connected}) > 1- \varepsilon$, and we are done.

Finally if
\[k+\lfloor C\log \frac{1}{\varepsilon} \rfloor< 0.6 \log n \qquad\textrm{and}\qquad \mathbb{P}(A_{k+\lfloor C\log (1/\varepsilon) \rfloor})\geq n^{-1.05},\] then since $\mathbb{P}(A_{k'})$ monotonically decreases as $k'$ increases we have $\mathbb{P}(A_{k'})\geq n^{-1.05}$ for every $k': k\leq k' \leq k+\lfloor C \log \frac{1}{\varepsilon} \rfloor$. Thus by Theorem~\ref{aksha}
we have, for all $k': k \leq k' \leq k+\lfloor C \log \frac{1}{\varepsilon} \rfloor$,
\[\mathbb{P}(A_{k'+L}) < e^{-1} \mathbb{P}(A_{k'}).\]
Since $k<0.6 \log n$, $\mathbb{P}(S_{n,k} \textrm{ connected})\leq 1- \varepsilon$ implies by Corollary~\ref{appimp} that $\mathbb{P}(A_k) \leq \frac{eM^2\log n}{n}\log \frac{1}{\varepsilon}$. Thus
\begin{align*}
\mathbb{P}(A_{k+\lfloor C \log \frac{1}{\varepsilon} \rfloor}) & \leq \exp \left( -\left\lfloor \frac{\lfloor C \log 1/ \varepsilon \rfloor}{L} \right\rfloor \right) \mathbb{P}(A_k) \\
& \leq \exp \left( -\left\lfloor \frac{C \log 1/ \varepsilon }{L} \right\rfloor \right) \left( \frac{eM^2\log n}{n} \log \frac{1}{\varepsilon}\right)\\
& \leq \exp \left( -\left\lfloor \frac{C \log 1/ \varepsilon }{L} \right\rfloor +1 +\log \log {1}/{\varepsilon} \right) \frac{M^2\log n}{n}. 
\end{align*} 

We now choose $C$: let
\[ C = \left(2+\frac{6}{\log 2}\right)L. \]
Since $\varepsilon \leq \frac{1}{2}$, we have that $\frac{\log 1/\varepsilon}{\log2}\geq 1$. Thus for this choice of $C$ we have
\begin{align*}
-\left\lfloor \frac{C \log 1/\varepsilon }{L}\right\rfloor &+1+ \log \log \frac{1}{\varepsilon}  \\
& \leq 2 + \log \log \frac{1}{\varepsilon} -\frac{C \log 1/\varepsilon}{L}\\
& = (2-2\frac{\log 1/\varepsilon}{\log 2})+(\log \log \frac{1}{\varepsilon}- \log \frac{1}{\varepsilon})- \frac{4 \log 1/\varepsilon}{\log 2} -\log \frac{1}{\varepsilon} \\
& \leq -4 - \log \frac{1}{\varepsilon}. 
\end{align*}

\noindent% 
Substituting this in the above bound for $\mathbb{P}(A_{k+\lfloor C \log {1}/{\varepsilon} \rfloor})$ we get 
\[\mathbb{P}(A_{k+\lfloor C \log 1 / \varepsilon \rfloor}) \leq \varepsilon \frac{M^2 \log n}{e^4n}.\]
By Corollary \ref{appimp}
this implies
\[\mathbb{P}(S_{n,k+ \lfloor C \log {1}/{\varepsilon} \rfloor} \textrm{ connected}) > 1-\varepsilon,\]
 proving the theorem.
\end{proof}

\section{Higher connectivity}

In this section, we shall apply our sharpness result
Theorem~\ref{maintheo} to prove Theorem~\ref{s-connected}, proving a
conjecture of Balister, Bollob\'as, Sarkar and
Walters~\cite{MR2479805}. Suppose that $\cP$ is any pointset in in the
square $S_n=[0,\sqrt n]^2$.  As before, let $G_k(\cP)$ denote the
$k$ nearest neighbour graph on $\cP$.

\begin{lemma}\label{l:main}
Suppose $S_{n,k}$ is the random geometric graph with $k$ an
integer lying between $0.3\log n$ and $0.6\log n$.  Let $s<0.1\log
n$. Then there is a constant $c_6$ such that
\[ 
\Prb(S_{n,k}\text{ not $s$-connected})\le c_6 (\log n )
\Prb(S_{n,k-1}\text { not $(s-1)$-connected})+O(n^{-3}).
\]
Moreover
\[ 
\Prb(S_{n,k}\text{ not $s$-connected})\le (c_6 \log n )^{s-1}
\Prb(S_{n,k-s+1}\text { not connected})+O\left(n^{-3}\log n\right).
\]
\end{lemma}

We shall need the following technical result to prove Lemma~\ref{l:main}.
\begin{lemma}\label{l:extra}
  Suppose $0.3\log n<k < 0.6\log n$. Then there exists $c_7$ such that
  the collection of pointsets $\cP$ from which we may delete at set $T$ of at most
  $0.1\log n$ points so that either of the following hold:
  \begin{itemize}
  \item there is any point $x\in S_n$ (not necessarily in $\cP$) with
    $\lceil 0.6\log n\rceil$-nearest neighbour radius in $\mathcal{P}\setminus S$ at least
    $c_7\sqrt{\log n}$
  \item $G_k(\cP)\setminus T$ contains at least two components of diameter at
    least $c_7\sqrt{\log n}$
  \end{itemize}
  has probability $O(n^{-3})$.
\end{lemma}
\begin{proof}
  This is an easy modification of Lemmas~2  and~6  of~\cite{MR2135151}
\end{proof}

\begin{proof}[Proof of Theorem~\ref{l:main}]
  We can view the Poisson distribution as follows. Suppose that $X_1,
  X_2, X_3, \ldots$ is an infinite sequence of uniformly distributed
  random variables in $S_n$ and let $Z\sim {\rm Po}(n)$. Then let the
  points in $\cP$ be given by $(X_i)_{i=1}^Z$.  Let $\cP_m$ denote the
  collection of pointsets with exactly $m$ points which we give the
  conditional measure which we shall sometimes denote $\Prb_m$.  From
  this point of view it is easy to see that we have $m$ measure
  preserving maps $\phi_i$ for $1\le i\le m$ from $\cP_m$ to
  $\cP_{m-1}$ given by deleting the point $X_i$.
  We shall
  usually abbreviate $\phi_{1}$ to $\phi$.

Let $\cA_s$ denote the collection of pointsets $\cP$ for which
$G_{k}(\cP)$ is not $s$-connected but $G_{k-1}(\cP)$ is
$(s-1)$-connected.  Let $\cB_s$ denote those pointsets $\cP$ for
which $G_{k-1}(\cP)$ is not $(s-1)$-connected.  Finally let $\cC$
denote the collection of pointsets $\cP$ for which either of the conditions in
Lemma~\ref{l:extra} hold, which we shall think of as the `bad'
pointsets. By Lemma~\ref{l:extra}, $\Prb(\cC)=O(n^{-3})$.

For any pointset $\cP$ in $\cA_s$ it is clear that (at least) one of
the functions $\phi_i$ maps $\cP$ into $\cB_s$.  Indeed, since
$G_k(\cP)$ is not $s$-connected, there is a point $X_i$ which we can delete
to make the graph not $(s-1)$-connected. Since $G_{k-1}(\cP\setminus
X_i)$ is a subgraph of $G_k(\cP)\setminus X_i$ the map $\phi_i$ is one
such function. Thus $\cA_s\subseteq \bigcup_{i=1}^m
\phi_i^{-1}(\cB_s)$.

  Note that $\Prb(|Z-n|>n/2)=o(e^{-n/2})$. We have
  \begin{align*}
    \Prb(\cA_s)
    &=\sum_{m=0}^{\infty} \Prb(\cA_s|Z=m)\Prb(Z=m)\\
    &=\sum_{m=n/2}^{3n/2} \Prb(\cA_s|Z=m)\Prb(Z=m)+o(e^{-n/2}) \\
    &= \sum_{m=n/2}^{3n/2} \Prb_m(\cA_s \setminus \cC)\Prb(Z=m)+O(n^{-3})\\
    &=\sum_{m=n/2}^{3n/2} \Prb_m(\cP\in\cA_s\setminus \cC \text{ and }\exists
    i:\phi_i(\cP)\in \cB_s)\Prb(Z=m) +O(n^{-3})\\ 
    &\le\sum_{m=n/2}^{3n/2}\sum_{i=1}^m \Prb_m(\cP\in\cA_s\setminus \cC \text{ and
    }\phi_i(\cP)\in \cB_s)\Prb(Z=m)+O(n^{-3})\\  
    &=\sum_{m=n/2}^{3n/2} m \Prb_m(\cP\in\cA_s\setminus \cC \text{ and
    }\phi(\cP)\in \cB_s)\Prb(Z=m)+O(n^{-3}).    
  \end{align*}

Now consider $\Prb_m(\cP\in\cA_s \setminus \cC \text{ and }
\phi(\cP)\in\cB_s)$.  For each $\cP\in\cA_s\setminus \cC$ with
$\phi(\cP)\in \cB_s$ we see that $G_{k-1}(\cP)$ is $(s-1)$-connected
but  $G_{k-1}(\phi(\cP))$ is
not $(s-1)$-connected.  Fix a separating set $T$ of $s-1$ vertices for
$G_{k-1}(\phi(\cP))$. Since $\cP\notin \cC$ we have that all but one of the
components in the separated graph $G_{k-1}(\phi(\cP))\setminus T$ are
small: less than $c_7\sqrt{\log n}$ in diameter. Fix one such
component $C$. Since $G_{k-1}(\cP)$ is $(s-1)$-connected we see that
$G_{k-1}(\cP)\setminus T$ is connected so $X_1$ must be joined to $C$
in $G_{k-1}(\cP)$ and, hence, that $X_1$ lies within distance
$c_7\sqrt{\log n}$ of $C$. Therefore $X_1$ lies within a set of
measure less than $4\pi c_7^2\log n$ which is determined by
$\cP\setminus X_1$. This event has probability less than
$\left(\frac{4\pi c_7^2 \log n}{n}\right)$. Thus, as $\phi$ is a
measure preserving transformation from $\Prb_m$ to $\Prb_{m-1}$,

\begin{align*}
\Prb_m(\cP\in\cA_s\setminus \cC\text{ and } \phi(\cP)\in\cB_s ) 
&\le \left(\frac{4\pi c_7^2 \log n}{n}\right) \Prb_{m}(\phi(\cP)\in \cB_s)\\
&= \left(\frac{4\pi c_7^2 \log n}{n}\right) \Prb_{m-1}(\cP\in \cB_s)
\end{align*}

To complete the proof note that $\Prb(Z=m)\le 2\Prb(Z=m-1)$ for all
$m>n/2$. Thus
\begin{align*}
\Prb(A_s)&\le\sum_{m=n/2}^{3n/2} m\Prb_m(\cP\in\cA_s\setminus \cC \text{ and
    }\phi(\cP)\in \cB_s)\Prb(Z=m)+O(n^{-3})\\
    &\le \sum_{m=n/2}^{3n/2}m\left(\frac{4\pi c_7^2 \log n}{n}\right) \Prb_{m-1}(\cP\in \cB_s) 
    \Prb(Z=m)+O(n^{-3})\\
    &\le \sum_{m=n/2}^{3n/2} \left(12\pi c_7^2 \log
n\right)\Prb_{m-1}(\cP\in \cB_s)\Prb(Z=m-1)+O(n^{-3})\\   
& \le \left(12\pi c_7^2 \log
n\right)\Prb(\cB_s) +O(n^{-3}).
\end{align*}

 Finally
observe that
\[
\{\cP\ :\ S_{n,k}\text{ not $s$-connected}\}\subseteq\cA_s\cup \cB_s
\]
so that the first part of the lemma holds with $c_6=12\pi c_7^2+1$: 
\[\mathbb{P}(S_{n,k} \textrm{ not $s$-connected}) \leq c_6 \log n \mathbb{P}(S_{n,k-1} \textrm{ not $(s-1)$-connected})+ O(n^{-3}).\]
Iterating this $s-1=O(\log n)$ times we obtain the second part of our claim.
\end{proof}

We can now finally turn to the proof of Theorem~\ref{s-connected}.
\begin{proof}[Proof of Theorem~\ref{s-connected}]
By Theorem~2 of~\cite{MR2479805} we may restrict ourselves to the case
where $s(n)$ is an integer sequence with $s(n) \leq \frac{\log
  n}{2\gamma\log\log n}$. Suppose that $k=k(n)$ is such that $S_{n,k}$
is connected \whp, so that
\[\Prb(S_{n,k}\text{ is not connected})\to 0.\] 
By Theorem~\ref{maintheo} with $\eps={(c_6\log n)}^{-s}$,
\[
\Prb(S_{n,k+\lfloor C\log 1 /\eps\rfloor} \text{ is not connected})<\eps
\]
for all sufficiently large $n$. (Explicitly, this is for all $n$ with
$n>\eps^{-\gamma}$. Given our choice of $\eps$ and the restriction on
$s$, $\eps^{-\gamma}$ is at most $\exp(\frac{1}{2}\log n +O(\frac{\log
  n}{\log\log n}))$, so that this is indeed satisfiable for large
enough $n$.) Now
\begin{align*}
C \log\frac{1}{\eps} +s-1&< 2Cs \log\log n\\
\end{align*} 
for all sufficiently large $n$. If $k+\lfloor2Cs\log\log
n\rfloor<0.6\log n$, we have by Lemma~\ref{l:main}
\begin{align*}
\Prb(&S_{n,k+\lfloor2Cs\log\log n\rfloor} \text{ not
  $s$-connected})\\
  &\leq {(c_6  \log n)}^{s-1} \Prb(S_{n,k-s+1 +\lfloor 2Cs \log \log n\rfloor} \textrm{ not connected}) +O\left(n^{-3}\log n\right)\\
  & \leq {(c_6  \log n)}^{s-1} \Prb(S_{n,k +\lfloor C \log 1 / \eps \rfloor} \textrm{ not connected}) +O\left(n^{-3}\log n\right)\\
  &< (c_6  \log n)^{s-1}\eps +O\left(n^{-3}\log  n\right)\\&= O(1/\log n)=o(1)
\end{align*}
as required. If on the other hand $k+\lfloor2Cs\log\log n\rfloor\geq
0.6\log n$, we have
\[\Prb(S_{n,k+\lfloor2Cs\log\log n\rfloor } \text{ is not $s$-connected})=o(1)\]
by Theorem~2 of~\cite{MR2479805}. The result follows.
\end{proof}

%\bibliography{mybib}{}
%\bibliographystyle{habbrv}

\end{document}